\documentclass[12pt]{amsart}

\usepackage{amssymb,amsmath,psfrag}

\newtheorem{lem}{Lemma}

\newtheorem{thm}[lem]{Theorem}

\theoremstyle{remark}

\newcommand\cl{\operatorname{cl}}

\newcommand\bbN{\mathbb{N}}

\begin{document}

\begin{center}

\Large{\bf Minimal non-orientable matroids in a projective plane}
\normalsize

{\sc Rigoberto Fl\'orez}
\footnote {The work of the first author was
performed at the State University of New York at Binghamton.}

\footnotesize
{\sc University of South Carolina Sumter\\
     Sumter, SC, U.S.A.\ 29150-2498}\\

\normalsize

and

{\sc David Forge}
\footnote {The work of the second author was performed while visiting
the State University of New York at Binghamton.}

\footnotesize
{\sc Laboratoire    de   recherche en   informatique UMR 8623 \\
 B\^at 490 Universit\'e Paris-Sud\\
91405 Orsay Cedex France} \\
{\tt forge@lri.fr}

\normalsize


\end{center}
\footnotesize {\it Abstract:} { We construct a new family of minimal
non-orientable matroids of rank three. Some of these matroids embed
in  Desarguesian projective planes. This answers a question of
Ziegler: for every prime power $q$, find a minimal non-orientable
submatroid of the projective plane over the $q$-element field.}

\normalsize

\thispagestyle{empty}

\section {introduction}
The study of non-orientable matroids has not received very much
attention compared with the study of representable matroids or
oriented matroids. Proving non-orientability of a matroid is known
to be a difficult problem even for small matroids of rank 3.
Richter-Gebert \cite{RG} even proved that this problem is
NP-complete. In the general case, there are only some necessary
conditions for a matroid to be non-orientable (see section 6.6 of
\cite{oriented}).

In 1991 Ziegler \cite {smnm} constructed a family of minimally non-orientable matroids of
rank three which are
submatroids of a projective plane over $\mathbb {F}_p$ for $p$ a prime.
These matroids are of size $3n+2$ with $n\ge2$ and the smallest is the
Mac Lane matroid on 8 elements (the only non-orientable matroid on 8 or fewer elements).
 Ziegler raised this question (\cite {oriented}, page 337):
For every prime power $q$, determine
a minimal non-orientable submatroid of the projective plane of order $q$.

We study an infinite family $\{F(n) : n \in \mathbb N \}$ of line arrangements
in the real projective plane (where $\mathbb N$ is the set of positive integers).
$F(n)$ consists of $2n+1$ lines constructed by taking the
infinite line together with a series of parallel lines going through two points.
We give an easy criterion to decide when it is possible to
extend the arrangement by  a pseudoline
passing through given  intersection vertices of $F(n)$.
This criterion gives a construction of a family of non-orientable
matroids with $2n+2$ elements for $n\ge3$. Our smallest example is again the Mac Lane matroid
but all others are different from Ziegler's arrangements.
Finally, we prove
that  a subfamily  of these non-orientable matroids embeds
in  Desarguesian projective planes coordinatized by fields of prime-power order.
This answers Ziegler's question.

The \emph {Reid cycle matroid} $R_{\text {cycle}}[k]$ for $k \geq 3$
is a certain single-element extension of our minimal non-orientable
matroid $M(n,\sigma){\mid (C\cup\{c_0,c_1\}})$ (given in  Theorem
\ref {t3}). Kung \cite [page 52]{jk} conjectured that for $k \geq
3$, the matroid $R_{\text {cycle}}[k]$ is non-orientable. McNulty
proved this conjecture \cite {{jmn}, {jn}}. Our Theorem \ref {t3}
shows that the Reid cycle matroid is not minimally non-orientable.

\section {Extension of pseudoline arrangements} \label{extension1}

We  define a family of  pseudoline
arrangements $F(n)$ of size $2n+1$ in the real projective plane.
We then study the possibility of extending such an $F(n)$
by new pseudolines going through  given sets of intersection vertices.

A pseudoline arrangement $L$ is a set of simple closed curves in the
real projective plane $\Pi$, of which each pair intersects at
exactly one point, at which they cross. Two arrangements are
\emph{isomorphic} if one is the image of the other by a continuous
deformation of the plane. An arrangement is \emph{stretchable} if it
is isomorphic to an arrangement of straight lines. The
\emph{extension} of an arrangement $L$ by a pseudoline $l$ is the
arrangement $L\cup l$ if the line $l$ meets correctly all the lines
of $L$. Given a finite set $V$ of vertices it is always possible to
draw a pseudoline going through the points of  $V$. However, given
an arrangement $L$ and a set  $V$ of points, it may be impossible to
construct an extension of $L$ by a pseudoline going through $V$.

We will use the following simple case of impossible extension.
Let $L=\{l_1,l_2\}$ be
an arrangement of two pseudolines meeting at a point $P_1$. These two lines
separate the real projective plane into two connected components
$C_1$ and $C_2$.
Let $P_2$ and $P_3$ be two points one in each of the two connected components
defined by $L$. Then there is no
extension of $L$ by a pseudoline going through the set of points
$\{P_1,P_2,P_3\}$.

Let $n$ be a positive integer.
We adopt the notation $[n]:=\{1,2,\ldots,n\}$.
Let $c_0$ be a line in the projective plane (in the affine representation
of the figures $c_0$ is the line at infinity),  and let $A$ and
$B$ be two points not on $c_0$. Let
$\{X_i : i\in [n]\}$ be a set of $n$  points of $c_0$ that appear in the order
$X_1, X_2,\ldots,X_n$ on $c_0$.
Let us call $F(n)$ a pseudoline arrangement with $2n+1$ pseudolines
$a_i$ for $ i \in [n]$, $b_i$ for $ i\in [n]$, $c_0$
such that

$$ \displaystyle { \bigcap_{i=1}^n a_i =A \text{ , } \ \ \bigcap_{i=1}^n  b_i =B,
\ \ \ \ \text{ and } \ \  a_i \cap b_i \cap c_0 = X_i} \ ,\ \forall i \in [n] .$$

Let us denote by $X_{i,j}$ the intersection point of the lines $a_{i}$ and $b_{j}$
for two different integers $i$ and $j$ (in this notation the point $X_i$ corresponds to $X_{i,i}$).

We remark that $F(n)$ is not uniquely defined but is unique up to isomorphism
(this is a key remark for the following).
Indeed, since the lines $a_i$ all meet at the vertex $A$, they also cross
there and nowhere else.
This gives then all the other crossings and their order on the lines. The points
 $A, X_{i,j}$ for $ j\in [n]$ appear on the line $a_i$ in  the order
\[ \big ( A, X_{i,1},X_{i,2},\ldots,X_{i,i-1},X_i,X_{i,i+1},\ldots,X_{i,n} \big ) \]
 and similarly the points  $B,X_{i,j}$ for $ i\in [n]$ appear on the line
$b_j$ in the order
\[ \big ( B, X_{1,j},\ldots,X_{j-1,j},X_j,X_{j+1,j},\ldots,X_{n,j} \big ).\]
$F(n)$ is stretchable; one can just  put $c_0$ at infinity and  take for the $a_i$
and $b_i$ $n$ pairs of parallel lines passing through  the two given points $A$ and $B$.
In fact, $F(n)$ is rational, i.e., it is isomorphic to an arrangement in the real
projective plane of lines defined  by equations with integer coefficients.
However, in the proofs we will not use the fact that $F(n)$ is stretchable or
rational and  for convenience in our figures we may represent $F(n)$  with pseudolines.

\begin{figure}[htpb]
\begin{center}
\includegraphics{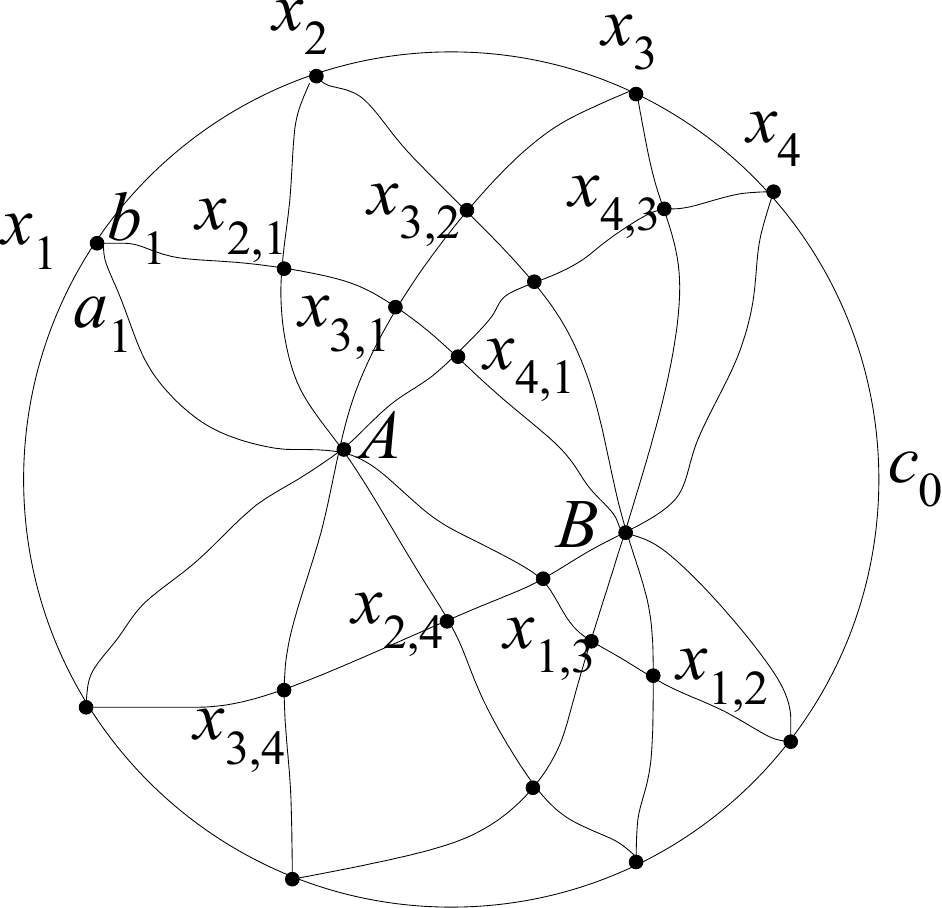}
\caption{The pseudoline arrangement $F(4)$.} \label{f1}
\end{center}
\end{figure}

\begin {lem} \label{lemmaorder}
For any integer $n\ge 3$ and any three  increasing integers $1\le
i_1<i_2<i_3\le n$, there exists an extension of the arrangement
$F(n)$ by a pseudoline passing through the three points
$X_{i_1,j_1}$, $X_{i_2,j_2}$, and $X_{i_3,j_3}$ if and only if
$j_1<j_2<j_3$ or $j_1>j_2>j_3$.
\end{lem}

\begin{proof}
We know the order in which  the  points $A,X_{i,j} $ for $ j\in [n]$
appear on the line $a_i$ and similarly the order in which the points
$B,X_{i,j}$ for $i\in [n]$ appear on the line $b_j$. The two lines
$a_i$ and $b_j$ meeting at $X_{i,j}$ separate the projective plane
into two connected components. Hence, point $X_{i,j}$ defines a
partition of the point set $S_{i,j}=\{X_{i',j'} : i' \not=i, j'
\not=j\}$ into the two parts
\[ S_{i,j}^+= \big \{ X_{i',j'} :  (i'-i)(j'-j)>0 \big \} \text { and } S_{i,j}^-
= \big \{ X_{i',j'} : (i'-i)(j'-j)<0 \big \}. \] There exists a
pseudoline passing through $X_{i_2,j_2}$ and the two other points
$X_{i_1,j_1}$ and $X_{i_3,j_3}$ if and only if $X_{i_1,j_1}$ and
$X_{i_3,j_3}$ belong to the same part of the partition defined by
$X_{i_2,j_2}$. Since we know that $i_1<i_2<i_3$, the last statement
is equivalent to the conclusion.
\end{proof}

\begin{lem}\label{alpha}
For any integer $n$ and any injective function  $f:D\rightarrow [n]$
where $D\subseteq[n]$, there exists an extension of the arrangement
$F(n)$ by a pseudoline $c_1$  passing through the  points
$X_{i,f(i)}$,  $i\in D$, if and only if the function $f$ is
increasing or decreasing.
\end{lem}

\begin{proof}
 The preceding lemma implies the conclusion.
\end{proof}

\begin{lem}\label{extension}
For any integer $n\ge2$  and for any cyclic  permutation $\alpha$ of
$[n]$, there exists an extension of the arrangement $F(n)$ by a
pseudoline passing through the points
$X_{\alpha^{i-1}(1),\alpha^i(1)}$, $ i\in [n]$, if and only if $n=2$
and $\alpha = (1 \ 2)$.
\end{lem}

\begin{proof}
If $n\ge3$ then Lemma \ref{alpha}  applied to $\alpha$ implies that
the bijection $\alpha$ is increasing or decreasing.  But a cyclic
permutation on more than two elements cannot be increasing or
decreasing. If $n=2$ then the only cyclic permutation is
$\alpha(1)=2$ and  $\alpha(2)=1$. And clearly one can find a
pseudoline passing through the two points $X_{1,2}$ and $X_{2,1}$
(in fact, any two points).
\end{proof}

\section{Orientability of matroids} \label{extension2}

From the Folkman-Lawrence Representation Theorem, the orientability of a
rank-three matroid is equivalent to its representability by a pseudoline
arrangement in the projective plane (see \cite{oriented} for more details).
In such a representation, the elements of the matroid correspond to pseudolines
of the arrangement. Similarly, the rank-two flats of the matroid correspond to
vertices of intersection of the pseudolines.
In this section we will define a family of minimal non-orientable matroids
using Lemma \ref{extension}.

Let  $A=\{a_i : i\in[n]\}$,   $B=\{b_i : i\in[n]\}$ and $\{c_0\}$
be disjoint sets.
For $i\in[n]$, let us call $X_i$ the set $\{a_i,b_i,c_0\}$.
Let  $M'(n)$ be the simple rank-3 matroid on the ground set
$E_n=A\cup B\cup \{c_0\}$ defined by the $n+2$ non-trivial rank-two flats :
$A$, $B$, and the $n$ sets $X_i$, $i\in[n]$.

Let $\tau$ be a permutation of $[n]$. The arrangement $F(n)$ was
defined in the previous section after placing the vertices $X_i$ in
the natural order on the line $c_0$. The position of the lines $a_i$
and $b_i$ and of the vertices $X_{i,j}$ were then constructed. Let
instead place the vertices $X_i$ on the line $c_0$ in the order
$X_{\tau(1)}$, $X_{\tau(1)}$, $\ldots, X_{\tau(n)}$. By keeping the
rule that $a_i$ and $b_i$ crosses on $c_0$ at vertex $X_i$, we get a
new pseudoline arrangement that we denote by $F(n,\tau)$.

\begin{lem}\label{perm}
The representations of $M'(n)$ by pseudoline arrangements are the
arrangements $F(n,\tau)$ where $\tau$ is a permutation of $[n]$.
\end{lem}

\begin{proof}
The permutation $\tau$ fixes the order of the points $X_i$ on the line $c_0$.
Once this order is fixed, every thing else
is determined by the fact that the lines $a_i$, for $i\in[n]$, go through
the points $X_i$ and $A$ and that the lines $b_i$, for $i\in[n]$, go through
the points $X_i$ and $B$.
\end{proof}

Let $\sigma$ be a permutation of $[n]$ without fixed elements.
We denote by $M(n,\sigma)$ the matroid extension of $M'(n)$ by an element $c_1$
such that the sets  $\{a_i,b_{\sigma(i)},c_1\}$, for $i\in [n]$ are
the additional non-trivial rank-2 flats.  This means that in $M(n,\sigma)$,
the new element $c_1$
is the intersection of the rank-2 flats $\cl(a_i,b_{\sigma(i)})$, $i\in [n]$.
Note that in the pseudoline
representation of the matroid, the pseudoline $c_1$ will have to pass through
the vertices of intersection $a_i\cap b_{\sigma(i)}$.
To the permutation $\sigma$ corresponds naturally the
bipartite graph $G_\sigma$ with vertex set $A\cup B$ and with edge set
 \[ \big \{ \{a_i,b_i\} : i\in [n] \big \}\cup \big
\{ \{a_i,b_{\sigma(i)}\} : i\in[n]\big \}. \]
In the graph $G_\sigma$, two vertices $a_i$ and $b_j$ form an edge
$\{a_i,b_j\}$ if
and only if  they both belong to some 3-point line with $c_0$ or $c_1$.
The graph $G_\sigma$ is clearly 2-regular, which implies that it
is a union of disjoint cycles. Let us point out that a $2k$-cycle of
the graph $G_\sigma$ corresponds to a $k$-cycle of the permutation $\sigma$.

\begin{figure} [htbp]
\begin{center}
\includegraphics{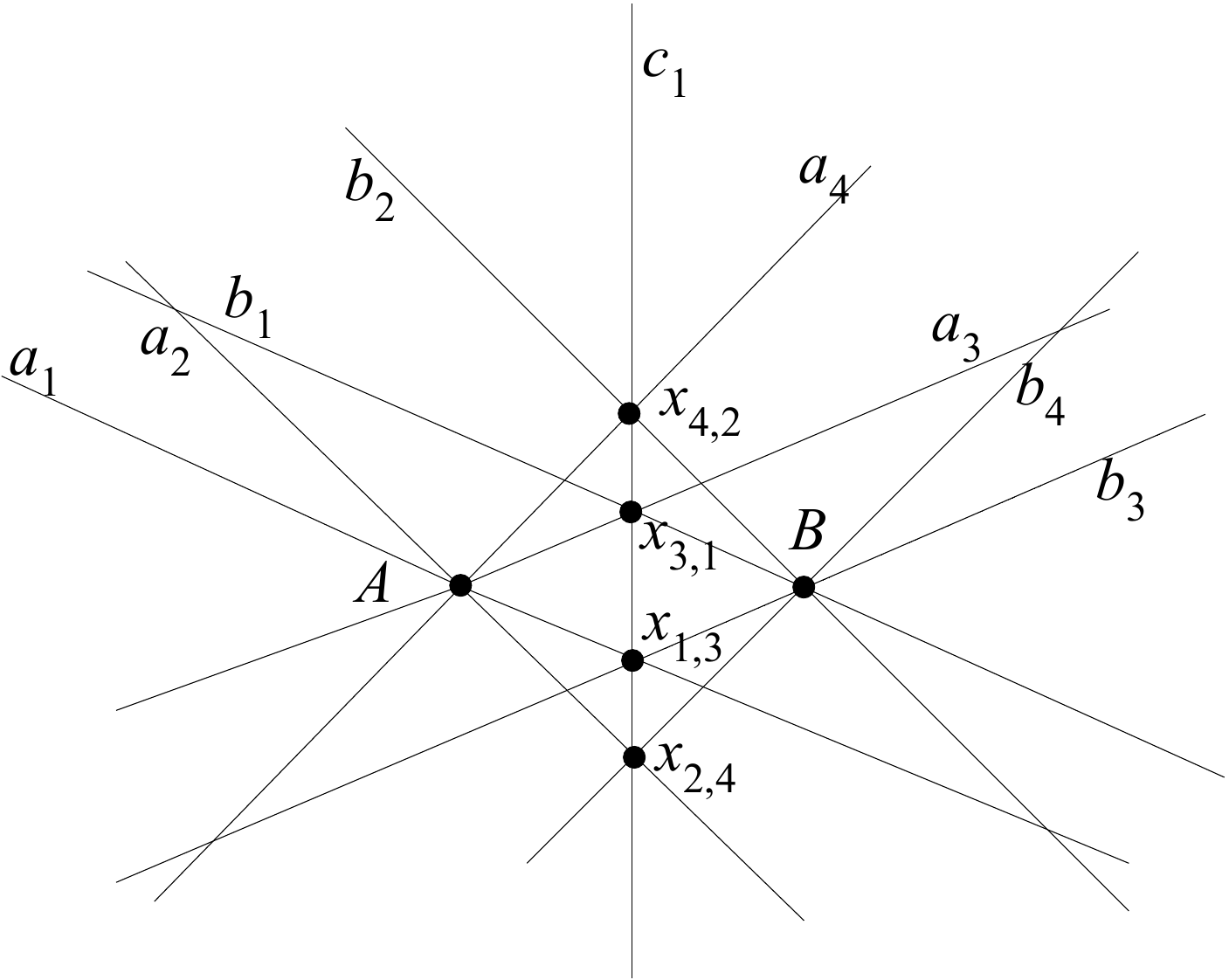}
\caption{ A linear realization of $M(4,(1 \ 3)(2 \ 4)) $}
\label{f2}
\end{center}
\end{figure}

\begin{thm}  \label{t3} Let $n\ge2$ and let $\sigma$ be a permutation of $[n]$
without fixed elements. The matroid $M(n,\sigma)$ is orientable
if and only if the graph $G_\sigma$  has no cycle of  length greater than four.
Moreover, if for some $k\ge 3$, the graph  $G_\sigma$ contains a cycle of
length $2k$,  say on the vertex set
\[ C=\big \{ a_i,a_{\sigma(i)},\ldots,a_{\sigma^{k-1}(i)},b_i,b_{\sigma(i)},
\ldots,b_{\sigma^{k-1}(i)} \big \}, \]
then the restriction $M(n,\sigma){\mid (C\cup\{c_0,c_1\}})$ is a minimal
non-orientable matroid.
\end{thm}

\begin{proof}
If the graph $G_\sigma$ has a decomposition into  cycles of length only 4
(hence $n$ must be even), we give
an explicit realization (see Figure \ref {f2}). We first relabel
the elements using a permutation
$\tau $ defining the position of the vertices $X_i$ at infinity.
This permutation is defined
by the following algorithm.

Start with $k=1$ and $S=[n]$. While $S \neq \emptyset$ do:

\hspace{1cm} a) let $i$ be the smallest element of $S$ and set $\tau(i)\leftarrow k$
and $\tau(\sigma(i))\leftarrow n+1-k$;

\hspace{1cm} b) put $k\leftarrow k+1$ and $S\leftarrow S\setminus \{i,\sigma(i)\}$.

The algorithm stops when the permutation $\tau$ has been completely defined
(i.e., when $S$ is finally empty, which will happen after $n/2$ steps). Put the points $A$ and $B$ at
$(-1,0)$ and $(1,0)$ respectively. Using the permutation $\tau$, the following realization works:

(a) the line $a_{\tau ^{-1}(i)} $ has equation $y=-ix-i$, for $i \leq n/ 2$,

(b) the line $b_{\tau ^{-1}(i)} $ has equation $y=-ix+i$, for $i \leq n / 2$,

(c) the line $a_{\tau ^{-1}(n-i+1)}$ has equation $y=ix+i$, for $i \leq n / 2$,

(d) the line $b_{\tau ^{-1}(n-i+1)}$ has equation $y=ix-i$, for $i \leq n / 2$,

(e) the line $c_0$ is at infinity,

(f) the line $c_1$ has equation $x=0$.

If $G_\sigma$ contains a cycle $C$ of length $2k\ge 6$ then
the matroid $M(n,\sigma)(\mid C\cup\{c_0,c_1\})$ is an extension of $M'(k)$ by the
element $c_1$. By Lemma \ref{perm}, a  representation of $M'(k)$ is a pseudoline
arrangement $F(k,\tau)$ for a permutation $\tau$. Then a representation
of $M(n,\sigma){\mid C\cup\{c_0,c_1\}}$ is an extension of $M'(k)$
by a pseudoline $c_1$ going through the points $X_{\tau(i),\tau(\sigma(i))}$, $i\in [n]$.
By Lemma \ref{extension}, this is impossible.

Let us now prove the minimality of $M(n,\sigma)\mid (
C\cup\{c_0,c_1\})$ as a non-orientable matroid. If one of the $c_i$
is deleted we get a matroid isomorphic to $M'(n)$, which is
orientable. If we delete one of the $a_i$ or one of the $b_i$  (say
$a_1$) then the matroid $M(n,\sigma)\mid C \setminus a_1$ is
realized by the following line arrangement in the real projective
plane:  put the points $A$ and $B$ at $(0,0)$ and $(1,0)$
respectively and

(a) the line $a_i $ has equation $x=ix$, for $2\le i\le k$,

(b) the line $b_i $ has equation $y=ix+1$, for $2\le i\le k$,

(c) the line $c_0$ is at infinity,

(d) the line $c_1$ has equation $y=1$. \qedhere
\end{proof}

\begin{figure}[htpb]
\begin{center}
\includegraphics{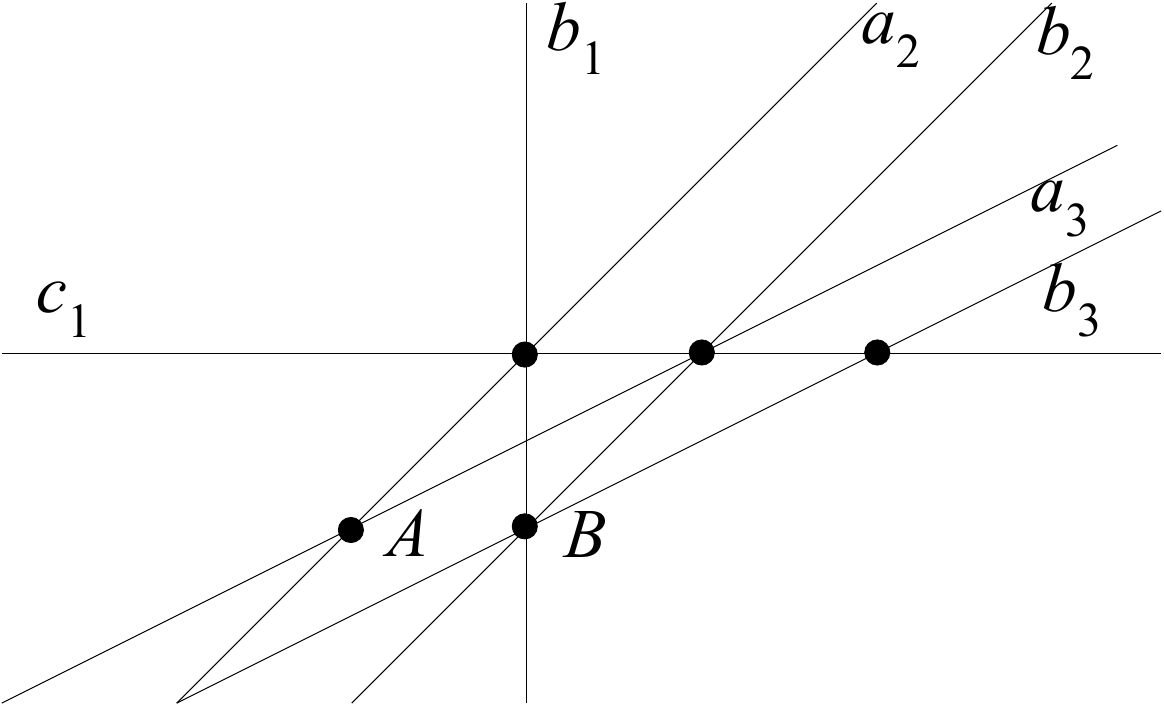}
\caption{  A linear realization of $M(3,(1 \ 2 \ 3))\setminus a_1$.} \label{f3}
\end{center}
\end{figure}

\section {Minimal non-orientable matroids contained in a projective plane} \label{last}

In this section we will define a  simple matroid $M(\mathfrak{G},
g_0,g_1 )$ where the definition of the lines depends on a given
group $\mathfrak{G}$ and two fixed elements of  $\mathfrak{G}$. We
will see that this matroid is a particular case of $M(n,\sigma)$.
The special case $M(\mathbb{Z}_n, 0, 1 )$ is a submatroid of a
non-orientable matroid given by McNulty \cite {jn}. If a finite
field $F$ contains $\mathfrak{G}$ as a multiplicative or an additive
subgroup then $M(\mathfrak{G}, g_0,g_1 )$ embeds in the projective
plane coordinatized by $F$. In Lemma \ref {bias6} (which follows by
Theorems 2.1 and 4.1 in \cite {b4}, because $M(\mathfrak{G}, g_0,g_1
)$ is a bias matroid of a gain graph) we prove this fact for finite
fields. With this lemma and Theorems \ref {mg} and \ref{Ziegler} we
will answer Ziegler's question.

Let $p^t$ be a prime power and let  $\mathbb{F}_{p^t}$ be a Galois
field. We will denote  by $\Pi_{p^t}$ the projective plane
coordinatized by $\mathbb{F}_{p^t}$. The points and lines of
$\Pi_{p^t}$ will be denoted by $[x,y,z]$ for $x,y,z$ in
$\mathbb{F}_{p^t}$, not all zero and $ \langle a ,b ,c
\rangle := \{ [x,y,z] : ax +by + cz = 0 \}$ for $a,b,c$ in
$\mathbb{F}_{p^t}$, not all zero.

Let $\mathfrak{G}$ be a finite group of order $n$ and let $g_0,g_1$
be two of its elements. Let $A=\{a_g : g\in\mathfrak{G}\}$, $B=\{b_g
: g\in\mathfrak{G}\}$ and $\{c_{g_0},c_{g_1}\}$ be disjoint sets.
Let $M( \mathfrak{G}, g_0,g_1 )$  be the simple matroid of rank 3 on
the ground set $E:= A\cup B \cup \{ c_{g_0}, c_{g_1} \}$ defined by
the $2n+2$ non-trivial rank-2 flats $A$, $B$, and the $2n$ sets $\{
a_{g}, b_{g \cdot g_0}, c_{g_0}  \}$, $g\in  \mathfrak{G}$, and $ \{
a_{g}, b_{g \cdot g_1}, c_{g_1} \}$, $g\in  \mathfrak{G}$.

\begin{thm} \label{mg} Let $g_0$ and $g_1$ be two different elements
of a finite abelian group $\mathfrak{G}$. Let $r$ be the order of
$g_0\cdot g_1^{-1}$. Then $ M( \mathfrak{G}, g_0,g_1 )$ is
non-orientable if and only if $r \geq 3$.
\end{thm}

\begin{proof} Let $n$ be the order of $\mathfrak{G}$. Let us first note that
$M( \mathfrak{G}, g_0,g_1 )$  is  isomorphic to an $M(n,\sigma)$.
Let $\alpha$ be a bijection from $[n]$ to $\mathfrak{G}$. Let
$\beta$ be the bijection from $[n]$ to $\mathfrak{G}$ defined by
$\beta(i)=\alpha(i)\cdot g_0$. Let   $\sigma$ be the permutation on
$[n]$ defined by $\sigma (i)= \beta^{-1}\big(\alpha(i)\cdot
g_1\big)$. The permutation $\sigma$  is clearly without fixed
elements. We now have an isomorphism $\phi$ between $M(n,\sigma)$
and $M(\mathfrak{G}, g_0,g_1 )$ given by $\phi(c_0) = c_{g_0}$,
$\phi(c_1) = c_{g_1}$, $\phi(a_i)=a_{\alpha(i)}$ and
$\phi(b_i)=b_{\beta(i)}$.

Let $G$ be the graph with vertex set $ \{ a_{g} : g \in \mathfrak{G}
\} \cup  \{ b_{g} : g \in \mathfrak{G} \} $ and edges  $\{ a_{g},
b_{g'} \}$ such that $ \{ a_{g}, b_{g'}, c_{g_0}  \}$ or $  \{
a_{g}, b_{g'}, c_{g_1}  \}$ is a line of  $ M( \mathfrak{G}, g_0,g_1
)$. This graph is the graph $G_\sigma$ for the corresponding
permutation $\sigma$.

A cycle of $G$ has the form
\[  \big \{  a_g ,  b_{g\cdot g_0} , a_{g \cdot g_0 \cdot g_1^{-1}} ,
b_{g\cdot g_0\cdot (g_1^{-1}\cdot g_0)} , a_{g\cdot (g_0\cdot
g_1^{-1})^2} ,\ldots , a_{g\cdot (g_0\cdot g_1^{-1})^{r-1}} ,
b_{g\cdot g_0 \cdot ( g_1^{-1} \cdot g_0)^{r-1}} \big \}.\]
Therefore the length of a cycle of $G$  is $2 r$. So, Theorem \ref
{t3} implies that $M( \mathfrak{G}, g_0,g_1 )$ is non-orientable if
and only if $r\ge3$.
\end{proof}

\begin{lem} \label{bias6} Let $p$ be a prime number and let $m\ge 2$ and $t\ge 1$ be two integers.

(i) $M( \mathbb{Z}_{p}, 0, 1 )$ embeds in $\Pi_p$.

(ii) If $m$ divides $p^t -1$, then $M(\mathbb{Z}_m,0, 1)$ embeds in $\Pi_{p^t}$.
\end{lem}

\begin {proof}[Proof of (i)]
Let $\psi$ be the map  from the ground set of $M( \mathbb{Z}_p, 0, 1
)$ into the point set of $\Pi_p$ defined as follows:
\[ \psi (a_i) = [0,i,1],\  \psi (b_i) = [1,i,1], \ \psi (c_0) = [1,0,0], \
\psi (c_1) = [1,1,0] \text { for } i \in  \mathbb{Z}_{p}. \]
By the definition of the incidence relation between points and lines in $\Pi_{p}$,

$ \big \{ [ 0,i, 1]: i \in \mathbb{Z}_p \big \} \subseteq \langle -1,0, 0 \rangle$,

$\big \{ [ 1,i,1]:  i \in \mathbb{Z}_p \} \subseteq  \langle -1,0,1 \rangle $,   and

$\big  \{ [1,0,0], [1,1,0] \} \subseteq  \langle 0,0,1 \rangle .$

Now, for fixed $i,j \in \mathbb{Z}_p$ and fixed $k \in \{0,1\}$, it is easy to verify that
$\psi( \{ a_i,b_j, c_k  \})$ is collinear in $\Pi_p$ if and only if $j=i + k$.
\end {proof}

\begin {proof}[Proof of (ii)]
Let $\phi $ be an isomorphism between the group $\mathbb {Z}_m$ and
the cyclic subgroup of order $m$ of the multiplicative group
$\mathbb{F}_{p^t}^*$ (such isomorphism exist because $m$ divides
$p^t-1$). Let $\psi$ be a map from the ground set of $M(
\mathbb{Z}_m, 0, 1 )$ into the point set of $\Pi_{p^t}$ defined as
follows:
$$ \psi (a_i) = [\phi (i),0,1], \ \psi (b_i) = [0,- \phi (i),1], $$
$$ \psi (c_0) = [1,\phi (0),0],\  \psi (c_1) = [1,\phi (1),0]  \text {  for } i \in \mathbb{Z}_{m}.$$

By the definition of the incidence  relation between points and
lines in $\Pi_{p^t}$,
$$\big \{ [\phi(i),0, 1]: i  \in \mathbb {Z}_{m} \big \} \subseteq \langle 0,1, 0 \rangle,$$
$$\big \{ [ 0,- \phi (i),1]:  i \in \mathbb {Z}_{m} \big \} \subseteq  \langle 1,0,0 \rangle , \text { and }$$
$$\big \{[1,\phi (0),0], [1,\phi (1),0] \big \} \subseteq  \langle 0,0,1 \rangle.$$

Now, for fixed $i,j \in \mathbb{Z}_m$  and fixed $k \in \{0,1\}$ it
is easy to verify that $\psi (\{ a_i,b_j, c_k \}) $ is collinear in
$\Pi_{p^t}$ if and only if $j=i+k$.
\end {proof}

\begin{thm} \label{Ziegler} Let $p \geq 3$ be a prime number and let
$m \geq 3$ and $t\ge 1$ be two integers.

(i) $M( \mathbb{Z}_p, 0, 1 )$  is a minimal non-orientable matroid that embeds in $\Pi_p$.

(ii) If $m $ is a divisor of $p^t -1 $ then $M( \mathbb{Z}_m, 0, 1 )$ is a minimal
non-orientable matroid that embeds in  $\Pi_{p^t}$.

(iii) $M( \mathbb{Z}_{p^t-1}, 0, 1 )$ is a minimal non-orientable matroid
that embeds in  $\Pi_{p^t}$ and in none of the $\Pi_{p^k}$ for $k < t$.

\end{thm}

\begin{proof} Parts (i) and (ii) follow by Theorem \ref {mg} and Lemma \ref {bias6}.

As a consequence of part (ii) $M( \mathbb{Z}_{p^t-1}, 0, 1 )$ is a
minimal non-orientable matroid in $\Pi_{p^t}$. Since $M(
\mathbb{Z}_{p^t-1}, 0, 1 )$ has a line with $p^t-1$ points, $M(
\mathbb{Z}_{p^t-1}, 0, 1 )$ does not embed in $\Pi_{p^k}$ for $k <
t$.
\end{proof}

The matroids given in parts $(i),$ $ (ii)$, and $(iii)$ of the
previous theorem are new minimal non-orientable matroids embeddable
in projective planes, except for $M( \mathbb{Z}_3, 0, 1 )$, which is
the Mac Lane matroid. Part  $(iii)$ answers Ziegler's question.

\section {Concluding remarks} \label{remarks}

At no moment in the previous sections did we really need to have a
finite set of points. We could have considered infinite rank-3
matroids and infinite pseudoline arrangements. For a permutation on
$\bbN$ without fixed elements, we can define the rank-3 infinite
matroid  $M(\bbN,\sigma)$ on the set $\{a_i :  i\in \bbN\}\cup\{b_i
: i\in \bbN\}\cup\{c_0,c_1\}$ by taking for its non-trivial rank-2
flats $A=\{a_i : i\in \bbN\}$, $B=\{b_i:i\in \bbN\}$,
$X_i=\{a_i,b_i,c_0\} $, $i\in \mathbb N$, and $\{
a_i,b_{\sigma(i)},c_1\}$, $i\in \mathbb N$.  The permutation
$\sigma$, as in the finite case, also defines a graph $G_\sigma$ on
the vertex set $A \cup B$. This graph is infinite but still of
degree 2. This implies that $G_\sigma$ is a union of cycles and
infinite 2-way paths. We then have the following results, which are
similar to Theorems \ref{t3} and \ref{mg}:

\begin{thm}   Let  $\sigma$ be a permutation of $\bbN$ without fixed elements.
The matroid $M(\bbN,\sigma)$ is orientable if and only if the graph
$G_\sigma$  has no cycle of  length greater than four.
\end{thm}

\begin{thm}
Suppose that  $ \mathfrak{G}$ is a finitely generated abelian group.
Then $ M( \mathfrak{G}, g_0,g_1 )$ is non-orientable if and only if
the order of $g_0\cdot g_1^{-1}$ is finite and greater than $2$.
\end{thm}

We want to remark also that $M(\mathbb{Z}_n, 0, 1 )$ is linearly
representable over the  complex numbers $\mathbb{C}$ (It follows by
\cite [Theorem 2.1] {b4}). Therefore, $M( \mathbb{Z}_n, 0, 1 )$
embeds in the projective plane coordinatized by  $\mathbb{C}$.

\section* {Acknowledgment}

We thank Thomas Zaslavsky for his helpful comments and valuable
advices.


\begin{thebibliography}{9}

\bibitem{oriented}
A.~Bj\"orner, M.~Las Vergnas, B. ~Sturmfels, N.~White, and
G.~Ziegler, \emph{Oriented Matroids}, second edition. Cambridge
University Press, Cambridge, Eng., 1999.

\bibitem{jk}
J.~Kung, Extremal matroid theory. \emph {Graph Structure Theory}
(Seattle, WA, 1991), pp.\ 21--61 Contemp. Math., vol.\ 147, Amer.
Math. Soc., Providence, RI, 1993.

\bibitem{jmn}
J. ~McNulty,
\emph{Hyperplane Arrangements and Oriented Matroids}.
Ph.D. Dissertation. University of North Carolina, 1993.

\bibitem{jn}
J. ~McNulty,
Two new classes of non-orientable matroids.
Unpubished manuscript, 1994.

\bibitem{RG}
J. ~Richter-Gebert,
Testing orientability for matroids is NP-complete.
\emph{Adv. in Appl. Math.} {\bf 23} (1999), no.~1,  78--90.

\bibitem{b4}
T.~Zaslavsky, Biased graphs. IV: Geometrical realizations. \emph{J.
Combin. Theory Ser. B} {\bf 89} (2003),  231--297.

\bibitem{smnm}
G. ~Ziegler,
Some minimal non-orientable matroids of rank three.
\emph{Geom. Dedicata} {\bf 38} (1991),  365--371.

\end{thebibliography}
\end{document}